\documentclass{siamart190516}

\usepackage{amssymb,url,amsmath,graphicx,mathdots,bm}
\usepackage{accents}
\usepackage{tikz}
\usetikzlibrary{shapes,arrows,positioning,decorations.pathreplacing,
  intersections,matrix,patterns,fadings,calc,decorations.markings,arrows.meta}

\usetikzlibrary{chains}
\usetikzlibrary{shapes.misc}
\usetikzlibrary{shapes.symbols}
\usetikzlibrary{shapes.geometric}
\usetikzlibrary{shapes.arrows}
\usetikzlibrary{fit}
\usetikzlibrary{shadows}
\usetikzlibrary{arrows}

\usepackage{pgfplots}
\usepackage{pgfplotstable}

\usepackage{subcaption}
\usepackage{listings}

\newsiamthm{remark}{Remark}
\DeclareMathOperator{\diag}{diag}
\newcommand{\nedelec}{N\'ed\'elec}

\title{On the convergence of monolithic multigrid for implicit Runge-Kutta time stepping of finite element problems}
\author{
  Robert C.~Kirby\thanks{Department of Mathematics, Baylor
    University; 1410 S.~4$^\text{th}$ St.; Waco, Texas 76706; Email: robert\_kirby@baylor.edu.  Supported by NSF 1909176.}}

\begin{document}

\maketitle

\begin{abstract}
  Finite element discretization of time dependent problems also require effective time-stepping schemes.
  While implicit Runge-Kutta methods provide favorable accuracy and stability problems, they give rise to large and complicated systems of equations to solve for each time step.
  These algebraic systems couple all Runge-Kutta stages together, giving a much larger system than for single-stage methods.
  We consider an approach to these systems based on \emph{monolithic} smoothing.  If stage-coupled smoothers possess a certain kind of structure, then the question of convergence of a two-grid or multi-grid iteration reduces to convergence of a related strategy for a single-stage system with a complex-valued time step.  In addition to providing a general theoretical approach to the convergence of monolithic multigrid methods, several numerical examples are given to illustrate the theory show how higher-order Runge-Kutta methods can be made effective in practice.
\end{abstract}


\section{Introduction}
Finite element methods (FEM) offer a wide range of robust and stable stable spatial discretizations of partial differential equations (PDE) that can yield high orders of accuracy and preserve underlying mathematical structure.  
Krylov methods, preconditioned by multigrid or other techniques, lead to highly efficient resolution of the underlying algebraic systems.
However, combining the high-order spatial discretization with comparable accuracy for evolution equations is a far less developed field.

\emph{Runge-Kutta} methods~\cite{butcher1996history,wanner1996solving} comprise a vast array of different time-stepping schemes including explicit and implicit methods.  Unlike multistep methods, Runge-Kutta schemes do not suffer from the famous Dahlquist barrier~\cite{dahlquist1963special} limiting A-stable methods to second order accuracy.
Certain families of methods also enforce critical stability properties related to conservation or dissipation.  
Many families of fully implicit collocation-type methods were studied theoretically in the past~\cite{butcher1964implicit}, but were deemed impractical owing to the large, stage-coupled algebraic systems required at each time step.
While diagonally implicit methods (DIRKs) such as~\cite{alexander1977diagonally} avoid the stage-coupled systems, most DIRKs have \emph{stage order} of only one or two and hence lose accuracy for the highly stiff systems generated by many time-dependent PDE.

The theoretical promise and higher stage accuracy of fully implicit methods remains tantalizing, and certain progress on efficient solvers has sparked new interest.
A common theme for such methods is to find some transformation or approximation of the stage-coupled system so that one can repurpose an effective technique for a single-stage method.
Early work in this direction, such as a block diagonal or triangular approximation to the system matrix, succeeds at reusing a single-stage solver, but the resulting preconditioners degrade with increasing number of Runge-Kutta stages~\cite{mardal2007order,staff2006preconditioning}.  
Rana \emph{et al} propose a somewhat different strategy~\cite{masud2021new}.  
Here, a triangular approximation of the Butcher matrix leads to a block triangular approximation of the stage-coupled system.
This approach not only enables re-use of effective single-stage solvers, but certain choices of the approximation to the Butcher matrix give excellent algorithmic scaling with respect to the number of RK stages.
Similar techniques have been used by Southworth \emph{et al} in~\cite{southworth2022fast1,southworth2022fast2} for the Navier-Stokes equations, and it is seen that good preconditioning can make fully implicit methods very competitive.

Here, we take a fundamentally different approach.
Rather than manipulating the overall system into a form where single-stage methods can be ably utilized, we propose multigrid smoothers by which the system can be solved \emph{monolithically}.
This approach, first suggested in~\cite{farrell2021irksome}, has yielded excellent empirical results for incompressible flow and magnetohydrodynamics~\cite{abu2022monolithic}, but the theory for such methods is quite sparse.
Early work in~\cite{vanlent2005} for finite difference spatial discretizations proposed and analyzed a block Jacobi smoother combining degrees of freedom for all the implicit stages at each grid point.
This approach was generalized to problems in $H(\mathrm{curl})$ in~\cite{TBoonen_etal_2009a}, where such a pointwise smoother was combined with an algebraic multigrid technique.

We give a significant extension of this approach to smoothing in the context of a family of geometric multigrid algorithms.
We introduce the concept of \emph{monolithicity}, in which the stage-coupled system decomposes to a set of independent systems for certain characteristic stages.
If a preconditioner/smoother for the coupled system respects this structure, then the multigrid algorithm similarly decomposes into characteristic stages.
The stage-coupled or monolithic multigrid algorithm converges if the underlying single-stage multigrid algorithm does.
While the decomposition yields complex-valued systems, this is only used for theoretical purposes -- real-valued systems are solved using only real arithmetic.
Consequently, our analysis provides a very general framework for extending single-stage multigrid to monolithic methods.

In Section~\ref{sec:setting}, we describe a family of evolution equations and Runge-Kutta time-stepping for their resulting Galerkin spatial discretizations. 
Then, we describe a general family of two-grid methods in Section~\ref{sec:2grid}.
We prove convergence of two-grid methods in terms of the two-grid convergence for the characteristic stages under assumptions on the smoother.
These assumptions seem somewhat abstract, but we also show that wide classes of single-stage smoothers can be adapted.
In particular, if the single-stage smoother can be derived in an additive Schwarz framework, then an analogous smoother can be constructed for the multi-stage case.  This framework includes the coupled smoothers in~\cite{TBoonen_etal_2009a,vanlent2005}.
The analysis of two-grid methods carries over immediately to a wide class of multigrid methods including V- and W-cycles, and we describe this briefly in Section~\ref{sec:mg}.
Finally, we give several numerical examples using monolithic multigrid as a preconditioner for GMRES in some model problems in Section~\ref{sec:apps} and conclusions in Section~\ref{sec:conc}.

\section{Problem setting}
\label{sec:setting}
Let $V$ and $W$ be Hilbert spaces with $V$ compactly embedded in $W$ and $T>0$ a real number.
We consider the abstract variational evolution equation of finding $u:(0, T] \rightarrow V$ such that 
  \begin{equation}
    \label{eq:theeq}
    m(u_t, v) + a(u, v) = F(t; v)
  \end{equation}
  for all $v \in V$, starting from some initial condition $u(0) = u_0 \in V$.
Here, $m(\cdot, \cdot)$ typically represents the $W$ inner product, but other choices are possible.  For example, letting $m(u_t, v)$ take the $W$-inner product of the projections into some subspace of $V$ allows us to consider differential-algebraic systems such as the time-dependent Stokes equations in the same framework.  Similarly, if $m$ takes derivatives of its arguments, then one can obtain Sobolev-type equations.
The bilinear form $a$ is bounded on $V \times V$, and $F:[0, T] \rightarrow V^\prime$.  We make no particular analytic assumptions such as coercivity at this point, other than to assume well-posedness of~\eqref{eq:theeq} and Galerkin approximations thereof.

To this end, we consider a finite-dimensional subspace $V_h \subset V$ equipped with a basis $\{ \phi_i \}_{i=1}^N$.  In the usual way, this leads to a discrete evolution equation, seeking $u_{h, t}: (0,T] \rightarrow V_h$ such that
\begin{equation}
  \label{eq:thediscreteeq}
  m(u_{h,t}, v_h) + a(u_h, v_h) = F(t; v_h)
\end{equation}
for all $v_h \in V_h$, starting from some $u_{h,0}$ suitably approximating the initial condition $u_0$ in $V_h$.

Before proceeding, we note at least two significant restrictions of our presentation, although generalizations are certainly possible.
First, we are working with conforming finite element methods, which admit naturally nested multigrid algorithms.  This allows us to focus on the stage-coupled structure.  Our convergence theory will show that single-stage multigrid convergence implies convergence for a kind of stage-coupled multigrid without making much use of the particulars of inter-grid transfer, so we expect these results to be readily applicable to more general multigrid settings.
Second, our formulation and analysis focuses on linear problems.
Nonlinear problems give rise to Jacobians with a similar, but more general, structure.  We also expect our convergence theory to hold provided that singe-stage multigrid works somehow uniformly over the underlying single-stage Jacobians.
Work in~\cite{abu2022monolithic} gives strong empirical support for this conjecture.

The discrete evolution equation~\eqref{eq:thediscreteeq} is equivalent to the system of ordinary differential (or differential algebraic) equations
\begin{equation}
  \label{eq:ODE}
  \mathrm{M} \mathrm{u}_t + \mathrm{K} \mathrm{u} = \mathrm{F},
\end{equation}
where
\begin{equation}
  \begin{split}
    \mathrm{M}_{ij} & = m\left( \phi_j, \phi_i \right), \\
    \mathrm{K}_{ij} & = a\left( \phi_j , \phi_i \right), \\
    \mathrm{F}(t)_i & = F(t)(\phi_i).
  \end{split}
\end{equation}

Now, we approximate the evolution of~\eqref{eq:ODE} by a Runge-Kutta method.
We partition $[0, T]$ into $N_t$ time steps of size $\Delta t = \tfrac{T}{N_t}$ and put $t^n = n \Delta t$.  Uniform time steps are purely a notational convenience and give no actual restriction for Runge-Kutta methods.
We seek approximations $\mathrm{u}^n \approx \mathrm{u}(t^n)$ to the ODE system, or equivalently, $u_h^n \approx u_h(t^n)$.

Runge-Kutta methods update the solution in terms of several \emph{stage variables}.  For an $s$-stage Runge-Kutta method starting from initial condition $\mathrm{u}_0$, the method is given by
\begin{equation}
  \label{eq:rkupdate}
  \mathrm{u}^{n+1} = \mathrm{u}^n + \Delta t \sum_{i=1}^s \mathrm{b}_i \mathrm{k}^i,
\end{equation}
where the stage vectors $\mathrm{k}^i \in \mathbb{R}^N$ collectively satisfy the algebraic system
\begin{equation}
  \label{eq:stagealg}
  \mathrm{M} \mathrm{k}^i + \Delta t \sum_{j=1}^s \mathrm{A}_{ij} \mathrm{k}^j = \mathrm{F}_i(t^n + \mathrm{c}_i \Delta t),  \ \ \ 1 \leq i \leq s.
\end{equation}
The numbers contained in $\mathrm{b}$, $\mathrm{c}$, and $\mathrm{A}$ are frequently arranged in a \emph{Butcher tableau}
\begin{equation}
  \label{eq:bt}
  \begin{array}{c|c}
    \mathrm{c} & \mathrm{A} \\ \hline
    & \mathrm{b}
  \end{array}.
\end{equation}
The structure of the algorithm is independent of the particular numerical values, but the stability and accuracy properties of the method depends strongly on them.

For general choices of $\mathrm{A}$, the algebraic system couples together all of the unknowns for all of the stages, resulting in a linear system that is $(sN) \times (sN)$ rather than just $N \times N$ for a single-stage method.
This is the practical tradeoff one makes for the favorable theoretical properties of fully implicit Runge-Kutta methods, and developing efficient solvers for this coupled system
 is the goal of this work.

We write the unknown stage variables in a single vector $\mathrm{k} \in \mathbb{R}^{sN}$ as
\begin{equation}
  \label{eq:stackk}
\mathrm{k} = \begin{bmatrix} \mathrm{k}^1 \\ \mathrm{k}^2 \\ \vdots \\ \mathrm{k}^s,
\end{bmatrix}
\end{equation}
and the forcing data for each stage in a vector $\mathrm{f}$ by
\[
\mathrm{f} = \begin{bmatrix} \mathrm{F}(t^n+\mathrm{c}_1\Delta t) \\
  \mathrm{F}(t^n+\mathrm{c}_2 \Delta t) \\
  \vdots \\
  \mathrm{F}(t^n+\mathrm{c}_s \Delta t)
  \end{bmatrix}.
\]
We can write out the linear system~\eqref{eq:stagealg} 
\begin{equation}
  \label{eq:stagealgkron}
  \left( \mathrm{I} \otimes \mathrm{M} + \Delta t \mathrm{A} \otimes \mathrm{K} \right) \mathrm{k} = \mathrm{f}
\end{equation}
where the identity matrix $\mathrm{I} \in \mathbb{R}^{s \times s}$ and $\otimes$ denotes the standard Kronecker product.
We note that even if $\mathrm{K}$ is symmetric, the coupled system will not be unless the Butcher matrix $\mathrm{A}$ is (which is quite rare).

This linear system~\eqref{eq:stagealg} is equivalent to a variational problem on a larger space.  We let $\mathbf{V}_h^s = \Pi_{i=1}^s V_h$ be the $s$-way Cartesian product of the finite-dimensional space $V_h$.  We seek $(k_{h, 1}, k_{h, 2}, \dots k_{h, s}) \in \mathbf{V}_h^s$ such that
\begin{equation}
  \left( k_{h,i}, v_{h, i} \right)
  + \Delta t \sum_{j=1}^s \mathrm{A}_{ij} a\left(k_{h, j}, v_{h, i} \right)
  = F\left(t+\mathrm{c}_i \Delta t; v_{h, i}\right),
\end{equation}

Just as $\mathbb{R}^N$ and $V_h$ are isomorphic through the identification of coefficients of the basis functions with vectors, so are
$\mathbb{R}^{sN}$ and $\mathbf{V}_h^s$.  To fix ideas, a member of $\mathbf{V}_h^ds$ consists of $s$ members of $V_h$, we can store all coefficients of the first function, followed by the second, and so on, just as we did in~\eqref{eq:stackk}.
Alternatively, one could store coefficients of each $\phi_i$ for all $s$ stages consecutively.  Such choices can impact performance, but we do not dwell on them in this work.

Following Butcher~\cite{butcher1976implementation}, when the matrix $A$ is invertible (which it is for most of our fully implicit families of interest), one can rewrite~\eqref{eq:stagealgkron} by multiplying through by $A^{-1} \otimes I$:
\begin{equation}
  \label{eq:stagealgkronbutch}
\left( A^{-1} \otimes M + \Delta t I \otimes K \right) \mathrm{k}
= (A^{-1} \otimes I) \mathrm{f} \equiv \tilde{\mathrm{f}},
\end{equation}
which has the advantage of making the typically stiff part of the matrix block diagonal.

Going forward, we define the matrix $\mathrm{B}$ to be
\begin{equation}
  \mathrm{B} = \left( \mathrm{I} \otimes \mathrm{M} + \Delta t \mathrm{A} \otimes \mathrm{K} \right).
\end{equation}
We will sometimes label $\mathrm{B}$ with subscripts indicating an approximating space in the multigrid hierarchy.  We also define a single-stage method by the pencil
\begin{equation}
  B_{z} = M + z K,
\end{equation}
where in our theory, $z$ may take on complex values.  This is only needed in our analysis and need not be actually computed.  We will include an additional subscipt such $B_{h,z}$ to distinguish between the single-stage methods on various spaces as needed.

\section{Monolithic two-grid methods}
\label{sec:2grid}
\subsection{Method formulation}
For a nonsingular system 
\begin{equation}
\label{eq:Cxb}
  C x = b
\end{equation}
representing a well-posed variational problem on $V_h$, one can define a simple iterative method as follows.  For some $W$ that somehow approximates $C$ but is simpler to invert, one take some initial $x^{(0)}$ and performs the iteration
\begin{equation}
  \begin{split}
    x^{(i+1)} & = x^{(i)} - W^{-1} \left( C x^{(i)} - b \right) \\
    & = \left( I - W^{-1} C \right) x^{(i)} + W^{-1} b \\
    & \equiv S x^{(i)} + W^{-1} b,
  \end{split}
\end{equation}
where the iteration matrix $S = I - W^{-1} C$ plays a critical role.  The method converges iff the spectral radius of $S$ is less than 1, and it is sufficient for convergence that $\| S \| < 1 $ in an induced matrix norm.
Examples of such iterations include the well known Jacobi iteration with
$W = \diag(C)$ or the Gauss-Seidel iteration with $W$ as the upper or lower-triangular part.

We note that several stages of the linear iteration can be combined
\begin{equation}
  x^{(i+\nu)} = S^\nu x^{(i)} + g^\nu,
\end{equation}
where $g^\nu = \sum_{j=0^{\nu-1}} S^{j} W^{-1} b$ is independent of the iterate $x^{(i)}$.

It is frequently the case for discrete PDE that such iterations converge quite slowly, and the rate deteriorates rapidly as the mesh is refined.  However, simple linear iterations play a critical role as \emph{smoothers} -- they eliminate high-frequency errors on the original mesh, and then the solution is corrected by solving an approximate problem on a coarser mesh.

The simplest way to describe and analyze such an approach is through a \emph{two-grid} method.
We suppose that $C$ discretizes a problem on $V_h$, and that we obtain some $C_H$ by discretizing the same problem over $V_H \subset V_h$, obtained on a coarser mesh.
One has a natural inclusion operator $\rho: V_H \rightarrow V_h$, and we associate with that a \emph{prolongation} matrix $P$ mapping vectors representing functions in $V_H$ to their representation as members of $V_h$.
Dual to this is a \emph{restriction} matrix $R$ somehow, frequently taken as the transpose of $P$.

Given a current iterate $x^{(i)}$ one obtains a two-grid method by first applying some $\nu$ steps of a smoothing iteration:
\begin{equation}
  \tilde{x}^{(i)} = S^\nu x^{(i)} + g^\nu.
\end{equation}
Then, one solves the coarse grid system
\begin{equation}
  C_H z = R \left( C \tilde{x}^{(i)} -b \right),
\end{equation}
which gives a kind of approximation to the error in $\tilde{x}^{(i)}$ on the coarse grid.  One then computes the next iterate of the two-grid method by prolonging this error approximation to the fine grid and updating the solution:
\begin{equation}
  x^{(i+1)} = x^{(i)} - P z.
\end{equation}
Combining these steps gives the two-grid iteration
\begin{equation}
  \label{eq:2grid}
  x^{(i+1)} = (I - P C_H^{-1} R A) S^\nu x^{(i)} + \tilde{g}
\end{equation}
for some suitably defined vector $\tilde{g}$ independent of $x^{(i)}$.

The two-grid iteration matrix
\begin{equation}
  T = (I - P C_H^{-1} R A) S^\nu
\end{equation}
defines the iteration and also the error propagation of the method.  The iteration~\eqref{eq:2grid} converges iff $\rho(T) < 1$.
As with simple preconditioned iteration, it is sufficient that $\| T \| < 1$ in some operator norm.

To apply this framework to obtain a two-grid method for the stage-coupled system~\eqref{eq:stagealgkron}, we first define restriction and prolongation operators for the multi-stage space:
\begin{equation}
  \label{eq:PR}
  \begin{split}
    \mathrm{P} & \equiv I \otimes P, \\
    \mathrm{R} & \equiv I \otimes R.
  \end{split}
\end{equation}
That is, we just prolong or restrict the degrees of freedom for each stage in the same way we would in a single-stage method.  We let $\mathrm{B}_h$ denote the orihinal system matrix obtained over $\mathbf{V}_h$, and $\mathrm{B}_H$ the system obtained on the coarse grid space $\mathbf{V}_H = \prod_{i=1}^s V_H$.

While restriction and prolongation operators  have a natural construction in terms of the finite element spaces, defining an appropriate smoother is more subtle.
We delve into this topic later, but for now, given some preconditioning matrix $\mathrm{W}$ and associated smoothing matrix
$\mathrm{S} = \mathrm{I} - \mathrm{W}^{-1} \mathrm{B}$,
one defines a two-grid method with iteration matrix
\begin{equation}
  \mathrm{T} = \left( \mathrm{I} - \mathrm{P} \mathrm{B}_H^{-1} \mathrm{R} \mathrm{B}_h \right) \mathrm{S}^\nu.
\end{equation}

\subsection{Analysis}
We begin by assuming that the Butcher matrix $A$ admits the eigenvalue decomposition
\begin{equation}
  \label{eq:Adiag}
  A = X \Lambda X^{-1},
\end{equation}
and we define the matrix
\begin{equation}
  \mathrm{X} = X \otimes I,
\end{equation}
which applies $X$ across the stages.  We also specify the identity matrix with which we take the Kronecker product of $X$ by
\begin{equation}
  \begin{split}
    \mathrm{X}_h & = X \otimes I_h, \\
    \mathrm{X}_H & = X \otimes I_H
  \end{split}
\end{equation}
when this is relevant to the context.

Critically, similarity transformations with $\mathrm{X}$ induce a block-diagonal structure for stage coupled systems, indicating a kind of decoupling into independent characteristic coordinates.  We give this idea a name:
\begin{definition}
  An $sN \times sN$ matrix $\mathrm{Y}$ is said to be \emph{monolithic} with respect to $\mathrm{X}$ if for $1 \leq i \leq s$ there exist $N \times N$ matrices $Y_i$ such that
  \begin{equation}
    \mathrm{X}^{-1} \mathrm{Y} \mathrm{X} = \diag_{1 \leq i \leq s} Y_i,
  \end{equation}
  and we just say that $\mathrm{Y}$ is \emph{monolithic} if the particular $\mathrm{X}$ is clear from context.
\end{definition}

\begin{remark}
  Equivalently, $\mathrm{Y}$ is monolithic with respect to $\mathrm{X}$ if it admits a decomposition of the form
  \begin{equation}
    \mathrm{Y} = \mathrm{X} \left( \diag_{1 \leq i \leq n} Y_i  \right) \mathrm{X}^{-1}.
  \end{equation}
\end{remark}

\begin{proposition}
  \label{prop:Bmon}
  The system matrix $\mathrm{B}_h$ is monolithic, as is $\mathrm{B}_H$.
\end{proposition}
\begin{proof}
  We just use the decomposition~\eqref{eq:Adiag}.  We omit the subscripts $h$ or $H$ so that
  $M$ and $K$ stand for $M_h$ or $M_H$ and $K_h$ or $K_H$, as needed -- the operations are the same in either case.
  \begin{equation}
    \label{eq:Bmon}    \begin{split}
      \mathrm{B} & = I \otimes M + \Delta t A \otimes K \\
      & = \left( X X^{-1} \right) \otimes M + \Delta t \left( X \Lambda X^{-1} \right) \otimes K \\
      & = \left( X \otimes I \right) \left[ I \otimes M + \Delta t \Lambda \otimes K \right] \left( X^{-1} \otimes I \right) \\
      & = \mathrm{X} \left[ \diag_{1\leq s} B_{\lambda_i \Delta t} \right] \mathrm{X}^{-1}.
    \end{split}
  \end{equation}
\end{proof}

Apparently, the multi-stage prolongation and restriction meet the definition of monolithicity as well, and furthermore:
\begin{lemma}
  \label{lem:PRX}
  The multi-stage prolongation and restriction operators given in~\eqref{eq:PR} satisfy
  \begin{equation}
    \label{eq:PRX}
    \begin{split}
      \mathrm{X}_h \mathrm{P} & = \mathrm{P} \mathrm{X}_H \\
      \mathrm{X}_h \mathrm{R} & = \mathrm{R} \mathrm{X}_H \\
    \end{split}
  \end{equation}
\end{lemma}
\begin{proof}
  This follows immediately from the definitions and properties of the Kronecker product.
\end{proof}

Our analysis requires that the preconditioner $\mathrm{W}$ is monolithic as well, namely that it admits the same kind of characteristic decomposition into a block diagonal matrix under $\mathrm{X}$.  
In the following subsection, we will look more closely at this and how to arrive at such preconditioners.

\begin{proposition}
  If the preconditioning matrix $\mathrm{W}$ is monolithic, then so is the smoothing matrix $\mathrm{S} = \mathrm{I} - \mathrm{W}^{-1} \mathrm{B}_h$.
\end{proposition}
\begin{proof}
  We suppose that
  \begin{equation}
    \mathrm{W} = \mathrm{X} \left( \diag_{1 \leq i \leq s} W_i \right) \mathrm{X}^{-1}.
  \end{equation}
  Then, using~\eqref{eq:Bmon} in Proposition~\ref{prop:Bmon}, we have
  \begin{equation}
    \label{eq:Smono}
    \begin{split}
      \mathrm{S} & = \mathrm{I} - \mathrm{W}^{-1} \mathrm{B}_h \\
      & = \mathrm{I} - \mathrm{X} \left( \diag_{1 \leq i \leq s} W_i \right)^{-1}
      \left( \diag_{1\leq s} M + \Delta t \lambda_i K \right) \mathrm{X}^{-1} \\
      & = \mathrm{X} \left( \mathrm{I} - \diag_{1 \leq i \leq s} W_i^{-1} B_{h, \lambda_i \Delta t} \right) \mathrm{X}^{-1} \\
      & = \mathrm{X} \left[ \diag_{1 \leq i \leq s} \left( I - W_i^{-1} B_{h, \lambda_i \Delta t} \right) \right] \mathrm{X}^{-1} \\
      & = \mathrm{X} \left( \diag_{1 \leq i \leq s} S_i \right) \mathrm{X}^{-1},
      \end{split}
  \end{equation}
  where
  \begin{equation}
    S_i \equiv I - W_i^{-1} B_{h, \lambda_i \Delta t}.
  \end{equation}
\end{proof}
Apparently, a monolithic smoother applies some smoother to each characteristic stage of $\mathrm{B}_h$.

\begin{proposition}
  \label{prop:2grid}
  If the preconditioner $\mathrm{W}$ is monolithic, then so is
  the iteration matrix~\eqref{eq:2grid} for the stage-coupled two-grid method.
\end{proposition}
\begin{proof}
  We use~\eqref{eq:Bmon} and~\eqref{eq:PRX} to write
  \begin{equation}
    \begin{split}
      \mathrm{P} \mathrm{B}_H^{-1} \mathrm{R} \mathrm{B}_h 
      & =  \mathrm{P}
      \mathrm{X} \left[ \diag_{1\leq s} B_{H, \lambda_i \Delta t}^{-1} \right] \mathrm{X}^{-1} \mathrm{R} \mathrm{X} \left[ \diag_{1\leq s} B_{h, \lambda_i \Delta t} \right] \mathrm{X}^{-1} \\
      & =  \mathrm{X}
      \mathrm{P} \left[ \diag_{1\leq s} B_{H, \lambda_i \Delta t}^{-1} \right]  \mathrm{R} \mathrm{X}^{-1} \mathrm{X} \left[ \diag_{1\leq s} B_{h, \lambda_i \Delta t} \right] \mathrm{X}^{-1} \\
      & = \mathrm{X} \left[ \diag_{1\leq s} P B_{H, \lambda_i \Delta t}^{-1} R \right]
      \left[ \diag_{1\leq s} B_{h, \lambda_i \Delta t} \right] \mathrm{X}^{-1}  \\
      & = \mathrm{X} \left[ \diag_{1\leq s} P B_{H, \lambda_i \Delta t}^{-1} R B_{h, \lambda_i \Delta t} \right] \mathrm{X}^{-1},
    \end{split}
  \end{equation}
  and hence
  \begin{equation}
    \mathrm{I} - \mathrm{P} \mathrm{B}_H^{-1} \mathrm{R} \mathrm{B}_h
    = \mathrm{X} \left[ \diag_{1\leq s} \left( I - P B_{H, \lambda_i \Delta t}^{-1} R B_{h, \lambda_i \Delta t} \right) \right] \mathrm{X}^{-1}.
  \end{equation}
  Combining this with~\eqref{eq:Smono} gives
  \begin{equation}
    \label{eq:Tsim}
    \begin{split}
      \mathrm{T} & = 
      \mathrm{X} \left[ \diag_{1\leq s} \left( I - P B_{H, \lambda_i \Delta t}^{-1} R B_{h, \lambda_i \Delta t} \right) S_i^\nu \right] 
      \mathrm{X}^{-1} \\
      & = \mathrm{X} \left[ \diag_{1\leq s} T_i \right]
      \mathrm{X}^{-1},
    \end{split}
  \end{equation}
  where we define
  \begin{equation}
    T_i \equiv \left( I - P B_{H, \lambda_i \Delta t}^{-1} R B_{h, \lambda_i \Delta t} \right) S_i^\nu.
  \end{equation}
  That is, the similarity transformation decomposes the two-grid iteration into a two-grid iteration applied to each of the stages.
\end{proof}
Since similarity transformations preserve eigenvalues,
\begin{theorem}
  \label{thm:2gridspec}
  The spectral radius of the two-grid iteration satisfies
  \begin{equation}
    \rho(\mathrm{T}) = \max_{1 \leq i \leq s} \rho(T_i),
  \end{equation}
  and so the two-grid iteration converges iff the two-grid iteration for each stage converges. 
\end{theorem}
The theoretical convergence of these iterations with possibly complex time steps determines the convergence of the actual calculation we propose.
We note that algebraic multigrid methods for single-stage equations are known to work well in these settings~\cite{TBoonen_etal_2009a,maclachlan2008algebraic}.

\subsection{A closer look at smoothers}
Now, we consider smoothers, working towards a general approach for constructing monolithic smoothers.
First, we consider a negative example - point Jacobi.
The point Jacobi preconditioner for $\mathrm{B}_h$ is just its diagonal:
\begin{equation}
\begin{split}
  \mathrm{W} & = \diag \left( I \otimes M_h + \Delta t A \otimes K_h \right) \\
  & = I \otimes \diag(M_h) + \Delta t \diag(A) \otimes \diag(K_h)
  \end{split}
\end{equation}
Now, we check whether similarity transformation under $\mathrm{X}$ produces a decoupled system:
\begin{equation}
    \mathrm{X}^{-1} \mathrm{W} \mathrm{X} =
    I \otimes \diag(M_h) + \Delta t \left( X \diag(A) X^{-1} \right) \otimes \diag(K_h).
\end{equation}
Although $X A X^{-1}$ is diagonal, $X \diag(A) X^{-1}$ will typically not be, so Theorem~\ref{thm:2gridspec} does not apply to the point Jacobi method.

However, a pointwise but stage-coupled smoothing, such as considered in~\cite{vanlent2005} for finite difference spatial discretizations with multi-stage time stepping, is monolithic.  
Consider the preconditioner
\begin{equation}
  \label{eq:bpjac}
  \mathrm{W} = I \otimes \diag(M_h) + \Delta t A \otimes \diag(K_h).
\end{equation}
As written, $\mathrm{W}$ is an $s \times s$ block matrix, with each block an $N \times N$ diagonal matrix.  This is readily reshuffled so that the inverse may be applied by solving $N \times N$ separate dense $s \times s$ systems of the form
\begin{equation}
  m_{ii} I + \Delta t k_{ii} A,
\end{equation}
where $m_{ii}$ and $k_{ii}$ are the diagonal entries of $M_h$ and $K_h$.
\begin{proposition}
  The preconditioner~\eqref{eq:bpjac} is monolithic.
\end{proposition}
\begin{proof}
  We compute the similarity transformation with $\mathrm{X}$:
  \begin{equation}
    \begin{split}
      \mathrm{X}^{-1} \mathrm{W} \mathrm{X}
      & = \mathrm{X}^{-1} \left( I \otimes \diag(M_h) + \Delta A \otimes \diag(K_h) \right) \mathrm{X} \\
      & = I \otimes \diag(M_h) + \Delta t \Lambda \otimes \diag(K_h) \\
      & = \diag_{1 \leq i \leq s} \left( \diag(M_h) + \Delta t \lambda_i \diag(K_h) \right),
    \end{split}
  \end{equation}
\end{proof}
So, the stage-coupled pointwise Jacobi preconditioner applies point Jacobi to each characteristic stage -- if point Jacobi is a good preconditioner for $B_{h, \lambda_i \Delta t}$, then we expect the monolithic two-grid method to converge.

Point Jacobi is not always the right smoother for the single-stage operators $B_{h, z}$.  More generally, additive Schwarz preconditioners include point Jacobi methods as well as many others.
Subspace decompositions based on vertex patches are frequently observed to give conditioning estimates independent of the degree of the underlying spatial discretization, and this can be rigorously proven at least for symmetric coercive operators~\cite{Pavarino:1993, Schoeberl:2008}.
Moreover, vertex patch spaces are essential tool for convergent multigrid methods in $H(\mathrm{div})$ and $H(\mathrm{curl})$~\cite{arnold2000multigrid} and also can be adapted to fluid problems such as Navier-Stokes and magnetohydrodynamics~\cite{adler2016monolithic, molenaar1991two}.
Here, we show a very general result -- given an additive Schwarz preconditioner constructed over some decomposition $V_h$, then an analogous decomposition for the product space $\mathbf{V}_h$ gives a monolithic additive Schwarz preconditioner.
The implication is that if a particular additive Schwarz decomposition is known for a single-stage operator, we expect the monolithic version of it to perform very similarly in the multi-stage case (module any issues arising in the convergence for complex time steps).  

Now, we allow a rather general setting for additive Schwarz methods, returning to~\eqref{eq:Cxb} over some space finite element space $V_h$.
We decompose $V_h$ into a (not necessarily direct) sum of subspaces $V_i \subset V_h$ by
\begin{equation}
  \label{eq:asdecomp}
  V_h = \sum_{i=1}^{N_S} V_i,
\end{equation}
and we let $P_i$ and $R_i$ denote the matrix representations of appropriate prolongation and restriction operators between $V_i$ and $V_h$.
Then, we operator by specifying its inverse:
\begin{equation}
  W^{-1} = \sum_{i=1}^{N_S} P_i C^{-1} R_i,
\end{equation}
which amounts to solving $N_S$ problems restricted to (typically small) subspaces.

Now, suppose that we have some monolithic $\mathrm{C}$ posed over the coupled space $\mathbf{V}_h$, and suppose that we have an effective additive Schwarz preconditioner for the underlying single-stage method.
We define the subspaces
\begin{equation}
  \mathbf{V}_i = \prod_{j=1}^s V_i
\end{equation}
as the $s$-way product of each of the spaces in the original decomposition, which induces a decomposition 
of $\mathbf{V}_h$ by
\begin{equation}
  \label{eq:blockasdecomp}
  \mathbf{V}_h = \sum_{i=1}^{N_s} \mathbf{V}_i.
\end{equation}

We let
\begin{equation}
  \begin{split}
    \mathrm{P}_i & = I \otimes P_i, \\
    \mathrm{R}_i & = I \otimes R_i
  \end{split}
\end{equation}
be the prolongation and restriction operators between $\mathbf{V}_h$ and $\mathbf{V}_i$ obtained by operating on each component separately.

This decomposition of the prolongation/restriction operators defines the inverse of an operator by
\begin{equation}
  \label{eq:scschwarz}
  \mathrm{W}^{-1} = \sum_{i=1}^{N_S} \mathrm{P}_i \mathrm{C}^{-1} \mathrm{R}_i.
\end{equation}

We have a commuting relationship between the subspace prolongation/restriction operators and the matrix $\mathrm{X}$ exactly analogous to Lemma~\ref{lem:PRX}.

\begin{theorem}
  If $\mathrm{C}$ is monolithic, then so is its stage-coupled Schwarz preconditioner $\mathrm{W}$~\eqref{eq:scschwarz}.
\end{theorem}
\begin{proof}
  Let $\mathrm{C}$ be monolithic, with
  \begin{equation}
    \mathrm{C} = \mathrm{X} \left( \diag_{1\leq i \leq s} C_i \right) \mathrm{X}^{-1},
  \end{equation}
  and then
  \begin{equation}
    \mathrm{C}^{-1} = \mathrm{X} \left( \diag_{1\leq i \leq s} C_i^{-1} \right) \mathrm{X}^{-1}
  \end{equation}
  
  Then,
  \begin{equation}
    \begin{split}
      \mathrm{X}^{-1} \mathrm{W}^{-1} \mathrm{X}_h & =
      \mathrm{X}^{-1} \left[ \sum_{j=1}^{N_S} \mathrm{P}_j \mathrm{C}^{-1} \mathrm{R}_j \right] \mathrm{X} = \sum_{j=1}^{N_S}  \left[ \mathrm{X}^{-1} \mathrm{P}_j  \mathrm{C}^{-1} \mathrm{R}_j  \mathrm{X} \right] \\
      & = \sum_{j=1}^{N_S}  \left[  \mathrm{P}_j \mathrm{X}^{-1} \mathrm{C}^{-1} \mathrm{X} \mathrm{R}_j \right] = \sum_{i=j}^{N_S}  \left[ \mathrm{P}_j \left( \diag_{1 \leq i \leq s} C_i^{-1} \mathrm{X}\right) \mathrm{R}_j \right] \\
      & = \sum_{i=1}^{N_S} \diag_{1 \leq j \leq s} P_j C_i^{-1} R_j = \diag_{1 \leq j \leq s} \left( \sum_{j=1}^{N_s} P_j C_i^{-1} R_j \right),
    \end{split}
  \end{equation}
  and so $\mathrm{W}^{-1}$ and hence $\mathrm{W}$ are monolithic.
This calculation shows that the stage-coupled additive Schwarz method amounts to applying additive Schwarz to each characteristic stage.
\end{proof}

\subsection{On the spectrum the Butcher matrix}
Our theory shows that the monolithic two-grid method convergence is equivalent to multigrid convergences of stages under the characteristic decomposition.
Since, the time steps for these stages turn out to be the eigenvalues of the Butcher matrix times the original time step, the eigenvalues of $A$ are important.
Figure~\ref{fig:eigs} shows the eigenvalue distribution for the Butcher matrix of Gauss-Legendre and RadauIIA methods for various numbers of stages.
Although the eigenvalues are complex, they decrease in size as we increase the number of stages.

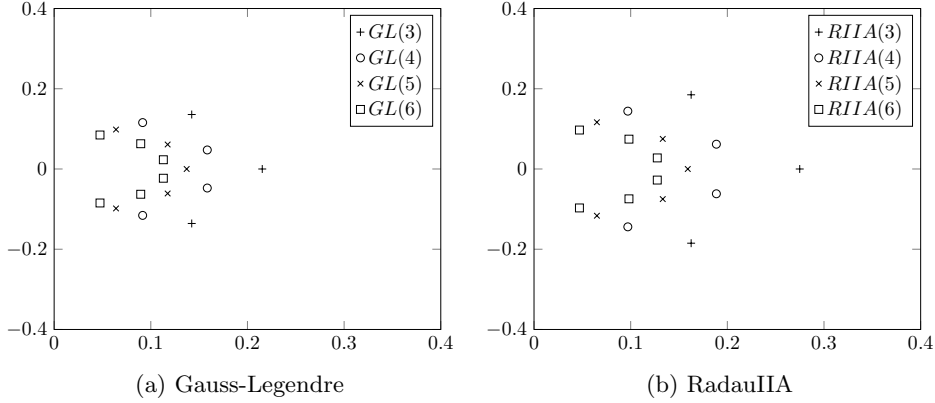
\begin{figure}
  \begin{center}
\begin{subfigure}[c]{0.49\textwidth}
  \begin{tikzpicture}[scale=0.75]
    \begin{axis}[xmin=0, xmax=.4, ymin=-0.4, ymax=0.4]
      \addplot[mark=+, only marks] table [x=Re, y=Im, col sep=comma] {GL3eigs.csv};
      \addlegendentry{$GL(3)$}
      \addplot[only marks, mark=o] table [x=Re, y=Im, col sep=comma] {GL4eigs.csv};
      \addlegendentry{$GL(4)$}
      \addplot[only marks, mark=x] table [x=Re, y=Im, col sep=comma] {GL5eigs.csv};
      \addlegendentry{$GL(5)$}
      \addplot[only marks, mark=square] table [x=Re, y=Im, col sep=comma] {GL6eigs.csv};
      \addlegendentry{$GL(6)$}
    \end{axis}
  \end{tikzpicture}
  \caption{Gauss-Legendre}
\end{subfigure}
\begin{subfigure}[c]{0.49\textwidth}
  \begin{tikzpicture}[scale=0.75]
    \begin{axis}[xmin=0, xmax=0.4, ymin=-0.4, ymax=0.4]
      \addplot[mark=+, only marks] table [x=Re, y=Im, col sep=comma] {RIIA3eigs.csv};
      \addlegendentry{$RIIA(3)$}
      \addplot[only marks, mark=o] table [x=Re, y=Im, col sep=comma] {RIIA4eigs.csv};
      \addlegendentry{$RIIA(4)$}
      \addplot[only marks, mark=x] table [x=Re, y=Im, col sep=comma] {RIIA5eigs.csv};
      \addlegendentry{$RIIA(5)$}
      \addplot[only marks, mark=square] table [x=Re, y=Im, col sep=comma] {RIIA6eigs.csv};
      \addlegendentry{$RIIA(6)$}                  
    \end{axis}
  \end{tikzpicture}
  \caption{RadauIIA}
\end{subfigure}
\end{center}
  \caption{Eigenvalues of the Butcher matrix $A$ for Gauss-Legendre and RadauIIA methods with various numbers of stages}
  \label{fig:eigs}
\end{figure}

While the behavior of the eigenvalues for increasing stage count is relatively benign, the eigenvalue matrix is less so.  A direct numerical calculation shows that the Butcher matrix is far from normal and hence, the eigenvector matrix cannot be unitary.
Moreover, its condition number grows exponential as the stage count is increased, as shown in Figure~\ref{fig:cond}.
This motivates our focus on eigenvalues rather than norms of iteration matrices for our two grid, and later, multigrid, methods.

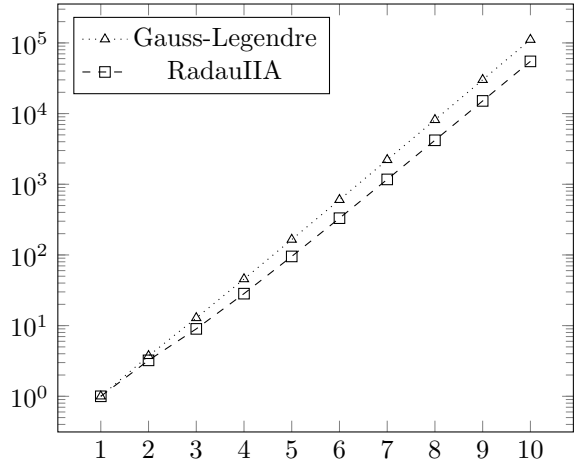
\begin{figure}
  \begin{center}
  \begin{tikzpicture}
    \begin{semilogyaxis}[xtick=data, log basis y={10}, legend pos=north west]
      \addplot[dotted, mark=triangle, mark options={solid}] table [x=k, y=kappa, col sep=comma] {GaussLegendrecond.csv};
      \addlegendentry{Gauss-Legendre};
      \addplot[dashed, mark=square, mark options={solid}] table [x=k, y=kappa, col sep=comma] {RadauIIAcond.csv};
      \addlegendentry{RadauIIA};      
    \end{semilogyaxis}
  \end{tikzpicture}
  \end{center}
  \caption{
    2-norm condition numbers of eigenvector matrix $X$ in the eigenvector decomposition $A = X \Lambda X^{-1}$ for the Butcher matrices of the $k$-stage RadauIIA and Gauss-Legendre methods. 
    }
\label{fig:cond}
\end{figure}

Consider equation~\eqref{eq:Tsim} showing the similarity of the monolithic two-grid iteration to the single-stage matrices resulting in Theorem~\ref{thm:2gridspec}.
Naively, taking norms of~\eqref{eq:Tsim} under the assumption that $\| T_i \| < 1$ for each characteristic stage results in a wildly pessimistic bound on $\| \mathrm{T} \|$ that grows exponentially with the stage count and does not imply convergence of the stage-coupled method.
It is known that the spectral radius \emph{eventually} controls the 2-norm convergence of an iterative process.
Theorem 1.2 of~\cite{olshanskii2014iterative} shows that, if $\rho(\mathrm{T}) < 1$, then, for large enough $i$, we have the 2-norm bound
\[
\| \mathrm{T}^i \|_2 \leq c_p i^{p-1} \rho(\mathrm{T})^{i-p+1},
\]
where $p$ is the size of the largest Jordan block in the decomposition $\mathrm{T} = \mathrm{Y} \mathrm{J} \mathrm{Y}^{-1} $.  The constant $c_p$ depends on the condition number of $\mathrm{Y}$, and we again have a pessimistic result.
While this may be unsatisfactory, we have observed no practical impacts on the behavior of monolithic multigrid methods.

\section{Monolithic multigrid}
\label{sec:mg}
Two-grid methods help to establish basic the theoretical structure, although in practice \emph{multigrid} methods approximate the coarse grid recursively, until some base coarse mesh is reached.  Frequently, one can use two-grid convergence to prove convergence of such multigrid methods~\cite{olshanskii2014iterative}.
Typically, these theorems require showing a kind of coarse grid approximation property and a certain norm estimate on the smoother.

Pursuing results of such flavor would introduce the condition number of the Butcher eigenvalue matrix.
Instead, we will continue in the spirit of our two-grid analysis, showing that a multigrid scheme with monolithic smoothing amounts to applying a multigrid to each characteristic stage.
Hence, convergence of the per-stage multigrid iterations (however that might be established) will imply convergence of the overall monolithic scheme, bypassing the conditioning of the characteristic decomposition.
The analysis is, despite technicalities, completely analogous to the two-stage case, and we summarize the approach here rather than giving a complete description of the process.

We pose a sequence of spaces $V_h^0 \subset V_h^1 \subset \dots \subset V_h^\ell$, and let $P_i$ and $R_i$ denote the matrices for prolongation and restriction between spaces $V_h^{i-1}$ and $V_h^i$.
For the stage-coupled system, we obtain the product spaces $\mathbf{V}_h^i$ and prolongation and restriction operators $\mathrm{P}_i = I \otimes P_i$ and $\mathrm{R}_i = I \otimes R_i$.  
We let $\mathrm{C}_i$ denote the monolithic operator on level $i$ and suppose that we have a monolithic preconditioner $\mathrm{W}_i$ for each $1 \leq i \leq \ell$.  Note that these prolongation and restriction operators are labeled with respect to the level of the multigrid hierarchy rather than to patches in an additive Schwarz decomposition.  The preconditioner $\mathrm{W}_i$ on each level may or may not itself be derived from an additive Schwarz-based method, but we are not using that structure at this point, only monolithicity.

These preconditioners then lead to smoothing operations
\begin{equation}
  \mathrm{S}_i = \mathrm{I} - \mathrm{W}_i^{-1} \mathrm{C}_i.  
\end{equation}

In the family of multigrid algorithms we consider, on each level, one applies some $\nu_1$ pre-smoothing iterations, and then restricts the residual to the next coarser mesh and recursively applies the multigrid method some $\gamma$ times.  Most frequently, one uses a so-called V-cycle with $\gamma=1$ or W-cycle with $\gamma=2$.  On the coarsest mesh, one exactly solves the system.  After recursion, the result is prolonged to the current mesh, and some $\nu_2$ post-smoothing iterations are performed.

Following~\cite{olshanskii2014iterative}, the iteration matrices for this family of multigrid algorithms satisfy (after some algebraic manipulation)
\begin{equation}
  \label{eq:mgT}
  \begin{split}
    \mathrm{T}_0 & = 0 \\
    \mathrm{T}_i & = \mathrm{S}_i^{\nu_1} \left( \mathrm{I} - \mathrm{P} \left( \mathrm{I} - \mathrm{T}_{i-1}^\gamma \right) \mathrm{C}_{i-1}^{-1} \mathrm{R}_i \mathrm{C}_i \right) \mathrm{S}_i^{\nu_2}.
  \end{split}
\end{equation}

Now, if we have a monolithic matrix defined on level $i$ of the multigrid hierarchy, we use a multiple subscripts to indicate the per-stage operators on each block, so that if $\mathrm{C}_i$ is a monolithic operator over space $\mathbf{V}^i_h$, then we have
\[
\mathrm{C}_i = \mathrm{X} \left( \diag_{j=1}^s C_{i,j} \right) \mathrm{X}^{-1},
\]
so the first subscript of $C_{i,j}$ refers to the multigrid hierarchy and the second to the block.  With this notation in hand, the following result is established inductively in the same way as done for the two-grid method:
\begin{proposition}
  If all the $\mathrm{W}_i$ are monolithic, then the multigrid method~\eqref{eq:mgT} is monolithic, satisfying
  \begin{equation}
    \mathrm{T}_i = \mathrm{X} \diag_{1 \leq j \leq s} \left[ S_j^{\nu_2} \left( I - P \left( I - T_{i-1, j}^\gamma \right) C_{i-1, j}^{-1} R C_{i, j} \right) S_j^{\nu_1} \right] \mathrm{X}^{-1}
  \end{equation}
\end{proposition}
Consequently,
\begin{theorem}
  The monolithic multigrid method converges if the underlying method converges for each characteristic stage.
\end{theorem}
Although multigrid methods can be used as iterations in their own right, they are also frequently used as preconditioners for a Krylov method such as GMRES.
If multigrid converges, then for a sufficient amount of smoothing its iteration matrix~\eqref{eq:mgT} has norm less of some $\xi$ less than $1$.  Using \cite[Theorem 1.33]{olshanskii2014iterative}, this controls the field of values and condition number of the preconditioned system and hence gives at least a linear convergence rate for GMRES.

\section{Applications}
\label{sec:apps}
Now, we consider a few model linear problems that, taken together, highlight the flexibility of our monolithic approach to smoothing.
In each case, we describe the underlying PDE, its finite element discretization, and a particular additive Schwarz smoother for the single-stage case.
Then, we test our multigrid multigrid method using the additive Schwarz smoother derived from the analogous decomposition for $\mathbf{V}_h^s$.

Our numerical results are obtained using the Irksome package~\cite{farrell2021irksome}, which provides Runge-Kutta methods on top of the Firedrake package~\cite{Rathgeber:2016}.  In each case, we report the iteration count and time required to solve the multistage linear system using eight M1 Max cores of a MacBook Pro with 64GB of RAM.

\subsection{Heat equation}
We pose the heat equation on the unit cube $\Omega = [0, 1]^3$:
\begin{equation}
  u_t - \Delta u = f, \\
\end{equation}
together with Dirichlet boundary conditions and some initial condition.
We generate a coarse mesh of $\Omega$ by dividing $\Omega$ into a $4 \times 4 \times 4$ array of cubes, then subdividing each cube into six tetrahedra in the standard way.
This can then be uniformly refined to create a multigrid hierarchy.

Spatial discretization by standard Galerkin finite elements of degree $r=1$ or $r=2$ leads to the variational evolution equation
\begin{equation}
  \left( u_{h, t} , v_h \right) + \left( \nabla u_h, \nabla v_h \right) = \left( f, v_h \right),
\end{equation}
together with appropriate initial conditions, and this can then be integrated in time with RadauIIA methods of various orders.

To demonstrate our monolithic multigrid technique, we fixed three levels of refinement for a total of 35,937 vertices.  For the single stage RadauIIA(1) (backward Euler) time-stepping scheme, we chose an additive Schwarz smoother based on a vertex patch decomposition~\cite{Schoeberl:2008}.
For each internal vertex $\mathbf{v}_i$ in the mesh, we let $\Omega^i$ be the closure of the star of $\mathbf{v}_i$ -- the set of all triangles of which $\mathbf{v}_i$ is a vertex.  Then, we take subspace $V_h^i$ to be set of all members of $V_h$ vanishing outside of $\Omega^i$.  By continuity, this enforces members of $V_h^i$ to vanish on $\partial \Omega^i$.
Typical patches for triangular meshes are shown in Figure~\ref{fig:heatpatches}; tetrahedral meshes are conceptually analogous but have many more cells per patch and are more difficult to visualize. 
When $V_h$ consists of $P^1$ functions, the patch subspaces have a single degree of freedom, and the additive Schwarz method reproduces point Jacobi smoothing.  When $V_h$ consists of $P^2$ functions, applying the patch smoother requires solving a small linear system for each internal vertex.

Now, such a decomposition of $V_h$ programmatically defines a decomposition of the stage-coupled space $\mathbf{V}_h^s$ -- one takes all of the stage degrees of freedom associated with the points in the patches shown.  For $P^1$ finite elements, this reproduces the pointwise block Jacobi smoother considered in~\cite{vanlent2005} for finite differences methods but gives a different method for $P^2$.
We implement the patch smoother using Firedrake's \lstinline{ASMStarPC} preconditioner.
This Python class extracts the degrees of freedom for each patch using PETSc's additive Schwarz framework, and allows us to use the \lstinline{tinyasm} package to solve all of the patch problems using optimized BLAS/LAPACK routines. 
On an $N \times N \times N$ mesh divided into tetrahedra, we let $h = \tfrac{1}{N}$ and set the time step of $\Delta t = \tfrac{\kappa}{h}$, with $\kappa = 1, 4, 8$.  Our results were very similar for each $\kappa$, we just report $\kappa=4$.

We solved the linear systems using PETSc's GMRES implementation to a relative Euclidean norm tolerance of $10^{-8}$.  The preconditioner was a multigrid V-cycle with two iterations of Chebyshev iteration with \lstinline{ASMStarPC} as a preconditioner on each level and coarse-grid problem solved directly via MUMPS.

\begin{figure}
  \begin{center}
    \begin{subfigure}[c]{0.45\textwidth}
      \centering
    \begin{tikzpicture}[scale=0.75]
    
    \draw[very thick] (0,0) -- (4,0) -- (4, 4) -- (0, 4) -- (0,0);
    \draw[very thick] (0,2) -- (4,2);
    \draw[very thick] (2,0) -- (2,4);
    \draw[very thick] (4,0) -- (0,4);
    \draw[very thick] (2,0) -- (0,2);
    \draw[very thick] (4,2) -- (2,4);
    \node[circle,draw,fill=gray!40,inner sep=0pt,minimum size=12pt] at (2,2) {};
    \tikzset{->-/.style={decoration={markings,mark=at position .6 with {\arrow[scale=1.7]{stealth}}},postaction={decorate}}}
    \end{tikzpicture}
    \label{p1}
    \caption{$P^1$ decomposition}
  \end{subfigure}
    \begin{subfigure}[c]{0.45\textwidth}
      \centering
    \begin{tikzpicture}[scale=0.75]
    
    \draw[very thick] (0,0) -- (4,0) -- (4, 4) -- (0, 4) -- (0,0);
    \draw[very thick] (0,2) -- (4,2);
    \draw[very thick] (2,0) -- (2,4);
    \draw[very thick] (4,0) -- (0,4);
    \draw[very thick] (2,0) -- (0,2);
    \draw[very thick] (4,2) -- (2,4);
    \node[circle,draw,fill=gray!40,inner sep=0pt,minimum size=12pt] at (2,1) {};
    \node[circle,draw,fill=gray!40,inner sep=0pt,minimum size=12pt] at (3,1) {};
    \node[circle,draw,fill=gray!40,inner sep=0pt,minimum size=12pt] at (1,2) {};
    \node[circle,draw,fill=gray!40,inner sep=0pt,minimum size=12pt] at (2,2) {};
    \node[circle,draw,fill=gray!40,inner sep=0pt,minimum size=12pt] at (3,2) {};
    \node[circle,draw,fill=gray!40,inner sep=0pt,minimum size=12pt] at (1,3) {};
    \node[circle,draw,fill=gray!40,inner sep=0pt,minimum size=12pt] at (2,3) {};    

    \end{tikzpicture}
    \label{p2}
    \caption{$P^2$ decomposition}
  \end{subfigure}
  \end{center}
  \caption{ASM decomposition for $P^1$ and $P^2$ discretizations on triangular meshes.}
  \label{fig:heatpatches}
\end{figure}
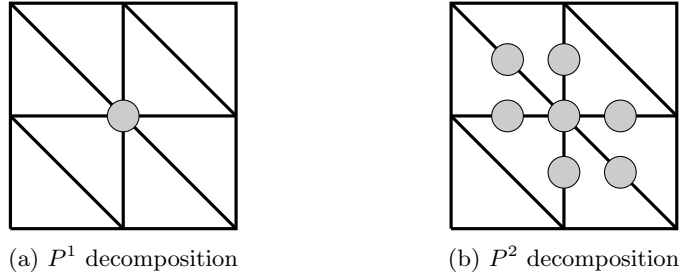

Figure~\ref{fig:heattime} reports the time taken to solve the linear system for one time step of a RadauIIA method with one through five stages, not counting the time to assemble the system matrix or preconditioner.  For linear elements, we observe almost no increase in run-time as we increase the number of stages.  For quadratic elements, the patches are considerably larger and so we notice more increase in run-time as the stages increase, but we observe that the five stage method is far less than five times as expensive as backward Euler.

\begin{figure}
  \begin{subfigure}[c]{0.48\textwidth}
    \begin{tikzpicture}[scale=0.7]
      \pgfplotstableread[col sep=comma,]{heat.deg1.cfl4.csv}\datatable
      \begin{axis}[ybar,ymin=0,xmax=5.4]
        \addplot table [x=stages, y=time]{\datatable};
      \end{axis}
    \end{tikzpicture}
    \caption{$P^1$}
  \end{subfigure}
  \begin{subfigure}[c]{0.48\textwidth}
    \begin{tikzpicture}[scale=0.7]
      \pgfplotstableread[col sep=comma,]{heat.deg2.cfl4.csv}\datatableq
      \begin{axis}[ybar,ymin=0,xtick=data,xmax=5.4]
        \addplot table [x=stages, y=time]{\datatableq};
      \end{axis}
    \end{tikzpicture}
    \caption{$P^2$}
  \end{subfigure}  
  \caption{Timing for solving the linear system for one time step of the heat equation using RadauIIA with $1 \leq s \leq 5$ stages.  With the $P^1$ discretization, solving the linear system with multigrid-preconditioned GMRES took 10 iterations for each RadauIIA method.  Only 8 iterations were required for the $P^2$ discretization.}
  \label{fig:heattime}
\end{figure}
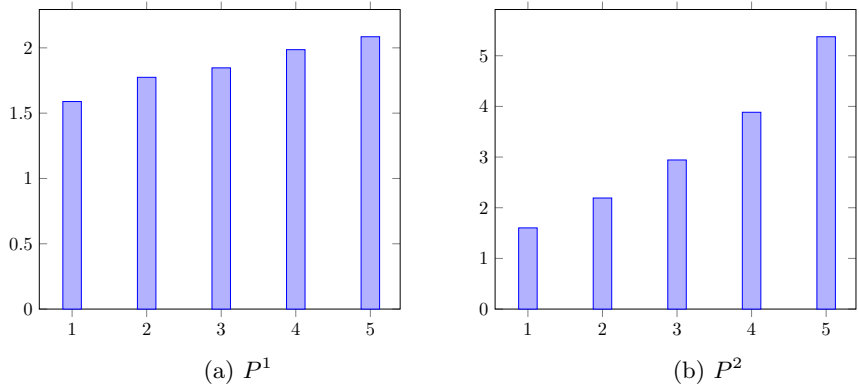

\subsection{Eddy current}
Next, we consider the diffusive eddy current problem on $H(\mathrm{curl})$.  
\begin{equation}
  E_t + \nabla \times (\nabla \times E) = 0,
\end{equation}
with natural boundary conditions.
We discretize this problem with first-kind \nedelec{} elements~\cite{nedelec1980mixed} of orders one and two on uniform refinements of the unit cube divided into a $4 \times 4 \times 4$ mesh partitioned into tetrahedra , giving rise to the variational problem
\begin{equation}
  \left(E_{h,t},v_h\right) +  \left(\nabla \times E_h, \nabla \times v_h\right) = 0,
\end{equation}
and we again integrate this problem in time with RadauIIA methods of various orders, using the same time steps as for the heat equation.
The resulting systems are larger than for the heat equation, and we consider only two levels of refinement, for a total of 4,913 vertices.

Algebraic multigrid methods were developed for multi-stage Runge-Kutta discretizations of this problem in~\cite{TBoonen_etal_2009a}.  These are based on algebraic multigrid for auxiliary space approach~\cite{hiptmair1998multigrid}, adapting the pointwise block smoothers from~\cite{vanlent2005}.
These methods used the underlying prolongation/restriction operators for the single-stage case to generate those for the multi-stage problem and used a block-type smoother.
Instead, we adopt the multigrid approach developed in~\cite{arnold2000multigrid}.  In this approach, one directly builds nested $H(\mathrm{curl})$ spaces on a hierarchy of meshes, prolonging and restricting in a natural way.  Even in the lowest-order case, a point Jacobi smoother fails, but an additive Schwarz smoother based on vertex patches is sufficient.
Figure~\ref{fig:nedpatch} shows an example patch for lowest-order edge elements on triangles.  A similar situation holds for tetrahedral edge elements, although the patch spaces for the second-order elements have much larger cardinality.  Each edge in the second-order space has two degrees of freedom, as does each face.  Supposing a tetrahedral patch associated with vertex $\mathbf{v}_i$ has 24 cells with 36 internal faces and 14 internal edges, the dimension of $V_h^i$ will be $2 \times 36 + 2 \times 14 = 100$.

\begin{figure}
  \centering
  \begin{tikzpicture}
    \tikzset{->-/.style={decoration={markings,mark=at position .6 with {\arrow[scale=1.7]{stealth}}},postaction={decorate}}}
    
    \draw[very thick] (10,0) -- (12,0);
    \draw[very thick] (12,0) -- (14,0);
    \draw[very thick,->-] (10,2) -- (12,2);
    \draw[very thick,->-] (12,2) -- (14,2);
    \draw[very thick] (10,4) -- (12,4);
    \draw[very thick] (12,4) -- (14,4);
    \draw[very thick] (10,0) -- (10,2);
    \draw[very thick] (10,2) -- (10,4);
    \draw[very thick,->-] (12,0) -- (12,2);
    \draw[very thick,->-] (12,2) -- (12,4);
    \draw[very thick] (14,0) -- (14,2);
    \draw[very thick] (14,2) -- (14,4);
    \draw[very thick] (12,0) -- (10,2);
    \draw[very thick,->-] (14,0) -- (12,2);
    \draw[very thick,->-] (12,2) -- (10,4);
    \draw[very thick] (14,2) -- (12,4);
  \end{tikzpicture}
  \caption{Patch for eddy current diffusion using lowest-order \nedelec{} elements on triangles.}
  \label{fig:nedpatch}
\end{figure}
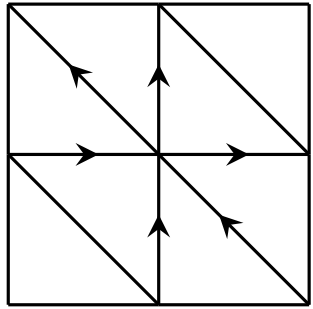

The right smoother in this case is based on vertex patches, and \lstinline{ASMStarPC} again does the right thin, and we can use a PETSc configuration nearly identical to the heat equation, except that we provided hand-tuned Chebyshev parameters to obtain some improvement the iteration counts.
As with the heat equation, Figure~\ref{fig:eddytime} reports the time taken to solve the linear system for one time step of a RadauIIA method with one through five stages, not counting the time to assemble the system matrix or preconditioner.  The patch smoother for the lowest-order method does not reduce to a pointwise block smoother as it does with $P^1$ for the heat equation, so we see somewhat more growth in run-time.  Still, the five-stage method required only about three times the run-time as the one-stage method.   This is roughly in line with the observations in~\cite{TBoonen_etal_2009a}.  For small matrices, the FLOP rate of BLAS implementations typically increases with the matrix size, and this largely explains the wildly superlinear behavior.

For larger matrices, the growth in the FLOP rate declines, and we begin to see this effect in the second-order discretization.
The patch problems are quite a bit larger (hundreds of unknowns), and we no longer have a $k$-stage method taking even less than $k$ times the run-time of backward Euler.  Still, for the accuracy obtained, we expect monolithic multigrid to make higher-order methods competitive.  
For example, the three-stage method takes a little more than five times the run-time of backward Euler.  We expect much better accuracy from one step of the formally fifth-order RadauIIA(3) than from five steps of backward Euler, so this should still lead to performance wins.

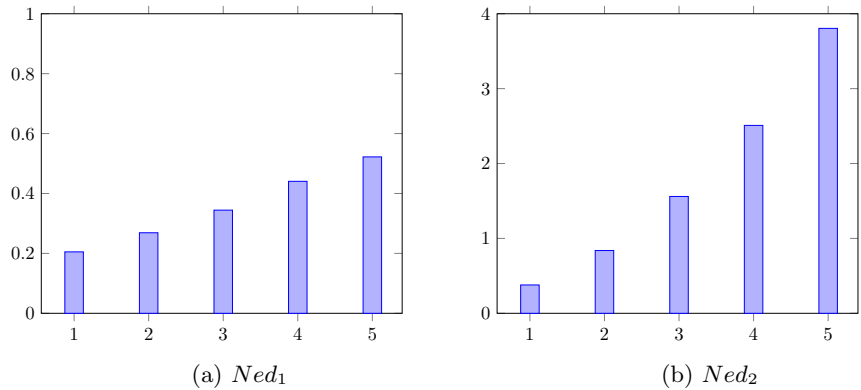
\begin{figure}
  \begin{subfigure}[c]{0.48\textwidth}
    \begin{tikzpicture}[scale=0.7]
      \pgfplotstableread[col sep=comma,]{eddy.deg1.cfl4.csv}\datatable
      \begin{axis}[ybar,ymin=0,ymax=1,xmax=5.4]
        \addplot table [x=stages, y=time]{\datatable};
      \end{axis}
    \end{tikzpicture}
    \caption{$Ned_1$}
  \end{subfigure}
  \begin{subfigure}[c]{0.48\textwidth}
    \begin{tikzpicture}[scale=0.7]
      \pgfplotstableread[col sep=comma,]{eddy.deg2.cfl4.csv}\datatableq
      \begin{axis}[ybar,ymin=0,ymax=4,xtick=data,xmax=5.4]
        \addplot table [x=stages, y=time]{\datatableq};
      \end{axis}
    \end{tikzpicture}
    \caption{$Ned_2$}
  \end{subfigure}  
  \caption{Timing for solving the linear system for one time step of the eddy current equation using RadauIIA with various stages.  With the lowest-order discretization, solving the linear system with multigrid-preconditioned GMRES took 10 iterations for each RadauIIA method.  Only 8 iterations were required for second-order discretization.}
  \label{fig:eddytime}
\end{figure}

\subsection{Stokes flow}
Our final example solves the time-dependent Stokes system for a fluid velocity $\mathbf{u}$ and pressure $p$
\begin{equation}
  \begin{split}
    \mathbf{u}_t - \Delta \mathbf{u} + \nabla p & = 0, \\
    \nabla \cdot \mathbf{u} & = 0.
  \end{split}
\end{equation}
We consider two-dimensional Stokes flow past a slightly off-center square obstacle, as shown in Figure~\ref{fig:stokesdomain}.  The domain consists of the rectangle $[0, 2.5] \times [0, 0.41]$ with the square $[0.15, 0.25] \times [0.15, 0.25]$ removed.  
No-flow boundary conditions are posed on the top and bottom edges and on the obstacle.  A horizontal parabolic profile is posed on the left edge, and natural boundary conditions are take on the right end.

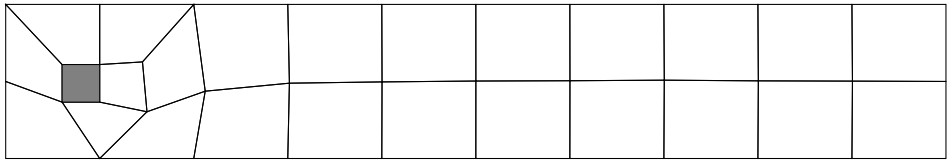
\begin{figure}
  \begin{tikzpicture}[scale=5.0]
\draw (0.0, 0.41) -- (0.0, 0.2050000000000006) -- (0.15, 0.15) -- (0.15, 0.25) -- cycle;
\draw (0.2500000000000022, 0.41) -- (0.25, 0.25) -- (0.15, 0.25) -- (0.0, 0.41) -- cycle;
\draw (0.0, 0.0) -- (0.0, 0.2050000000000006) -- (0.15, 0.15) -- (0.2499999999999995, 0.0) -- cycle;
\draw (0.3636984844536987, 0.2569990302222783) -- (0.3756511446533219, 0.1243488829171065) -- (0.25, 0.15) -- (0.25, 0.25) -- cycle;
\draw (0.5000000000000044, 0.41) -- (0.2500000000000022, 0.41) -- (0.25, 0.25) -- (0.3636984844536987, 0.2569990302222783) -- cycle;
\draw (0.2499999999999995, 0.0) -- (0.15, 0.15) -- (0.25, 0.15) -- (0.3756511446533219, 0.1243488829171065) -- cycle;
\draw (0.5000000000000044, 0.41) -- (0.3636984844536987, 0.2569990302222783) -- (0.3756511446533219, 0.1243488829171065) -- (0.5304984509564666, 0.179095479542571) -- cycle;
\draw (0.4999999999999989, 0.0) -- (0.2499999999999995, 0.0) -- (0.3756511446533219, 0.1243488829171065) -- (0.5304984509564666, 0.179095479542571) -- cycle;
\draw (0.5000000000000044, 0.41) -- (0.5304984509564666, 0.179095479542571) -- (0.754313380039602, 0.2004091668781849) -- (0.750000000000004, 0.41) -- cycle;
\draw (0.4999999999999989, 0.0) -- (0.5304984509564666, 0.179095479542571) -- (0.754313380039602, 0.2004091668781849) -- (0.7499999999999984, 0.0) -- cycle;
\draw (1.000000000000003, 0.41) -- (1.000642647303453, 0.2037960131723511) -- (0.754313380039602, 0.2004091668781849) -- (0.750000000000004, 0.41) -- cycle;
\draw (0.9999999999999978, 0.0) -- (0.7499999999999984, 0.0) -- (0.754313380039602, 0.2004091668781849) -- (1.000642647303453, 0.2037960131723511) -- cycle;
\draw (1.000000000000003, 0.41) -- (1.250000000000003, 0.41) -- (1.250115858936184, 0.2061154558189412) -- (1.000642647303453, 0.2037960131723511) -- cycle;
\draw (0.9999999999999978, 0.0) -- (1.249999999999997, 0.0) -- (1.250115858936184, 0.2061154558189412) -- (1.000642647303453, 0.2037960131723511) -- cycle;
\draw (1.500000000000002, 0.41) -- (1.250000000000003, 0.41) -- (1.250115858936184, 0.2061154558189412) -- (1.500258739161926, 0.2069030663463586) -- cycle;
\draw (1.499999999999997, 0.0) -- (1.500258739161926, 0.2069030663463586) -- (1.250115858936184, 0.2061154558189412) -- (1.249999999999997, 0.0) -- cycle;
\draw (1.500000000000002, 0.41) -- (1.750000000000002, 0.41) -- (1.751074964249576, 0.2084040981363664) -- (1.500258739161926, 0.2069030663463586) -- cycle;
\draw (1.499999999999997, 0.0) -- (1.500258739161926, 0.2069030663463586) -- (1.751074964249576, 0.2084040981363664) -- (1.749999999999996, 0.0) -- cycle;
\draw (2.000000000000001, 0.41) -- (2.001150491756403, 0.206608398473017) -- (1.751074964249576, 0.2084040981363664) -- (1.750000000000002, 0.41) -- cycle;
\draw (1.999999999999996, 0.0) -- (1.749999999999996, 0.0) -- (1.751074964249576, 0.2084040981363664) -- (2.001150491756403, 0.206608398473017) -- cycle;
\draw (2.000000000000001, 0.41) -- (2.001150491756403, 0.206608398473017) -- (2.251407945575682, 0.2059019540315802) -- (2.25, 0.41) -- cycle;
\draw (1.999999999999996, 0.0) -- (2.249999999999998, 0.0) -- (2.251407945575682, 0.2059019540315802) -- (2.001150491756403, 0.206608398473017) -- cycle;
\draw (2.25, 0.41) -- (2.251407945575682, 0.2059019540315802) -- (2.5, 0.2049999999999993) -- (2.5, 0.41) -- cycle;
\draw (2.249999999999998, 0.0) -- (2.5, 0.0) -- (2.5, 0.2049999999999993) -- (2.251407945575682, 0.2059019540315802) -- cycle;
\draw[fill=gray] (.15, .15) rectangle ++(.1, .1);
\end{tikzpicture}
  \caption{Coarse quadrilateral mesh of domain for Stokes flow around an obstacle.}
\label{fig:stokesdomain}
\end{figure}

Nothing in our framework restricts us to simplicial meshes; for this problem, we decompose our domain into quadrilateral meshes.  We discretize the velocity with continuous $Q^2$ elements (the images of reference biquadratic elements under a non-affine bilinear mapping) and pressure with discontinuous linear polynomials.

This discretization is used in~\cite{john2001higher}, where a cell-based Vanka-type smoother~\cite{vanka1986block} is developed for discontinuous pressure approximations.  A typical patch is shown in Figure~\ref{fig:vanka}, where we take all degrees of freedom attached to a given cell.  The continuity of velocities creates overlap between the subspaces, while pressure degrees of freedom only appear in a single patch space.

\begin{figure}
  \centering
  \begin{tikzpicture}[scale=0.75]
    \foreach \i in {0,2,4,6}
    {\draw[very thick] (\i,0) -- (\i,6);
      \draw[very thick] (0,\i) -- (6,\i);}
    \foreach \i in {2,3,4}
    \foreach \j in {2,3,4}
             {\node[circle,draw,fill=gray!40,inner sep=0pt,minimum size=12pt] at (\i,\j) {};}
             \node[circle,draw,fill=black,inner sep=0pt,minimum size=8pt] at (2.5,2.5) {};
             \node[circle,draw,fill=black,inner sep=0pt,minimum size=8pt] at (3.5,2.5) {};
             \node[circle,draw,fill=black,inner sep=0pt,minimum size=8pt] at (3,3.5) {};
  \end{tikzpicture}
  \caption{Vanka patch for Stokes discretization, with grey circles indicating velocity degrees of freedom and smaller black circles indicating pressures.}
  \label{fig:vanka}
\end{figure}
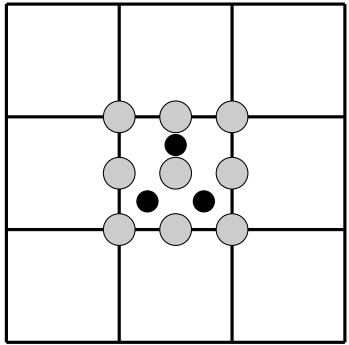

We constructed an initial coarse quadrilateral mesh of the domain using \lstinline{gmsh}~\cite{geuzaine2009gmsh}.
Our computational mesh of 25,024 vertices and 24,576 cells was constructed by uniform refinements of this initial mesh.
Then, we solved the linear systems arising from RadauIIA methods with one through five stages.  As before, we used GMRES with a relative tolerance of $10^{-8}$.  Multigrid V-cycles were used with a Vanka-type smoother on each level, and the coarse grid problem was solved with MUMPS.
The timings for solving a single linear system are reported in Figure~\ref{fig:stokestime}, where we see favorable performance for higher-order methods.  For example, the five-stage method required only about three times the run-time of backward Euler.

\begin{figure}
  \centering
  \begin{tikzpicture}[scale=0.7]
    \pgfplotstableread[col sep=comma,]{stokes.8procs.csv}\datatable
    \begin{axis}[ybar,ymin=0,xmax=5.4]
      \addplot table [x=stages, y=time]{\datatable};
    \end{axis}
  \end{tikzpicture}
  \caption{Timing for compute one time step of the Stokes equation with $Q^2$ velocity and discontinuous $P^1$ pressure using RadauIIA with various numbers of stages. RadauIIA(1) required 7 iterations, RadauII(2) required 6 iterations, and the rest of the cases only required 5.}
  \label{fig:stokestime}
\end{figure}
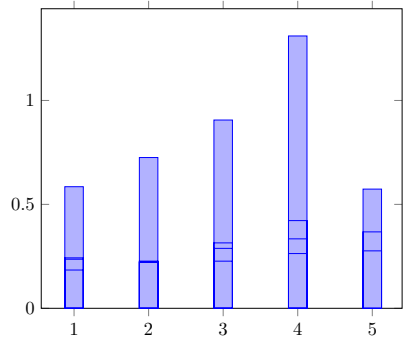

\section{Conclusions}
\label{sec:conc}
We have presented a general framework for developing and analyzing monolithic multigrid methods for stage-coupled systems arising in Runge-Kutta methods for finite element discretizations of time-dependent problems.
Given appropriate structure in the smoother, the resulting method converges iff the underlying method does for each characteristic stage.
This framework applies to many kinds of PDE.
Such theory, together with empirical results presented here and elsewhere in the literature, show that monolithic multigrid algorithms can be a powerful tool in realizing the full potential of fully implicit Runge-Kutta methods in practice.

At the same time, many avenues remain for future research.
First, the single-stage multigrid theory is often worked out for real-valued problems.
Despite experimental evidence, establishing such convergence in the context of complex-valued time steps seems to be open and requires techniques suitable for non-Hermitian problems.
Second, we have focused on conforming finite element discretization of linear problems, and extensions of the theory to other discretizations and to nonlinear problems is yet to be done.
Third, the additive Schwarz framework covers a wide range of possible monolithic smoothers.  We have not addressed multiplicative Schwarz or other more advanced kinds of smoothers but expect them to also prove monolithic.
Finally, stage-coupled smoothers lead to quite heavyweight local computations.  While these benefit from modern architectures, they may become expensive in the context of three-dimensional multiphysics applications.  Hence, practical work at accelerating these calculations may be necessary for fully large-scale calculations.

\section*{Acknowledgments}
The author thanks Prof.~Scott MacLachlan (Memorial University of Newfoundland) for many helpful discussions, especially regarding the use of the \lstinline{ASMStarPC} and \lstinline{ASMVankaPC} to obtain the numerical results.

\bibliographystyle{siamplain}
\bibliography{references}

\begin{thebibliography}{10}

\bibitem{abu2022monolithic}
{\sc R.~Abu-Labdeh, S.~MacLachlan, and P.~E. Farrell}, {\em Monolithic
  multigrid for implicit {R}unge-{K}utta discretizations of incompressible
  fluid flow}, Journal of Computational Physics,  (2023),
  \url{https://doi.org/10.1016/j.jcp.2023.111961}.
\newblock In press.

\bibitem{adler2016monolithic}
{\sc J.~H. Adler, T.~R. Benson, E.~C. Cyr, S.~P. MacLachlan, and R.~S.
  Tuminaro}, {\em Monolithic multigrid methods for two-dimensional resistive
  magnetohydrodynamics}, SIAM Journal on Scientific Computing, 38 (2016),
  pp.~B1--B24.

\bibitem{alexander1977diagonally}
{\sc R.~Alexander}, {\em Diagonally implicit {Runge--Kutta} methods for stiff
  {ODE}s}, {SIAM Journal on Numerical Analysis}, 14 (1977), pp.~1006--1021.

\bibitem{arnold2000multigrid}
{\sc D.~N. Arnold, R.~S. Falk, and R.~Winther}, {\em Multigrid in
  {$H(\mathrm{div})$} and {$H(\mathrm{curl})$}}, Numerische Mathematik, 85
  (2000), pp.~197--217.

\bibitem{TBoonen_etal_2009a}
{\sc T.~Boonen, J.~{Van lent}, and S.~Vandewalle}, {\em An algebraic multigrid
  method for high order time-discretizations of the div-grad and the curl-curl
  equations}, Applied Numerical Mathematics, 59 (2009), pp.~507--521,
  \url{https://doi.org/10.1016/j.apnum.2008.03.004}.

\bibitem{butcher1964implicit}
{\sc J.~C. Butcher}, {\em Implicit {R}unge-{K}utta processes}, Mathematics of
  computation, 18 (1964), pp.~50--64.

\bibitem{butcher1976implementation}
{\sc J.~C. Butcher}, {\em On the implementation of implicit {R}unge--{K}utta
  methods}, BIT Numerical Mathematics, 16 (1976), pp.~237--240.

\bibitem{butcher1996history}
{\sc J.~C. Butcher}, {\em A history of {R}unge-{K}utta methods}, Applied
  numerical mathematics, 20 (1996), pp.~247--260.

\bibitem{dahlquist1963special}
{\sc G.~G. Dahlquist}, {\em A special stability problem for linear multistep
  methods}, BIT Numerical Mathematics, 3 (1963), pp.~27--43.

\bibitem{farrell2021irksome}
{\sc P.~E. Farrell, R.~C. Kirby, and J.~Marchena-Menendez}, {\em {Irksome:
  Automating Runge--Kutta time-stepping for finite element methods}}, ACM
  Transactions on Mathematical Software, 47 (2021), pp.~1--26.

\bibitem{geuzaine2009gmsh}
{\sc C.~Geuzaine and J.-F. Remacle}, {\em {Gmsh: A 3-D finite element mesh
  generator with built-in pre-and post-processing facilities}}, International
  Journal for Numerical Methods in Engineering, 79 (2009), pp.~1309--1331,
  \url{https://doi.org/10.1002/nme.2579}.

\bibitem{hiptmair1998multigrid}
{\sc R.~Hiptmair}, {\em Multigrid method for {M}axwell's equations}, SIAM
  Journal on Numerical Analysis, 36 (1998), pp.~204--225.

\bibitem{john2001higher}
{\sc V.~John and G.~Matthies}, {\em Higher-order finite element discretizations
  in a benchmark problem for incompressible flows}, International Journal for
  Numerical Methods in Fluids, 37 (2001), pp.~885--903.

\bibitem{maclachlan2008algebraic}
{\sc S.~P. MacLachlan and C.~W. Oosterlee}, {\em Algebraic multigrid solvers
  for complex-valued matrices}, SIAM Journal on scientific computing, 30
  (2008), pp.~1548--1571.

\bibitem{mardal2007order}
{\sc K.-A. Mardal, T.~K. Nilssen, and G.~A. Staff}, {\em Order-optimal
  preconditioners for implicit {R}unge--{K}utta schemes applied to parabolic
  {PDE}s}, SIAM Journal on Scientific Computing, 29 (2007), pp.~361--375.

\bibitem{molenaar1991two}
{\sc J.~Molenaar}, {\em A two-grid analysis of the combination of mixed finite
  elements and {V}anka-type relaxation}, in {Multigrid Methods III}, Springer,
  1991, pp.~313--323.

\bibitem{nedelec1980mixed}
{\sc J.-C. N{\'e}d{\'e}lec}, {\em Mixed finite elements in $\mathbb{R}^3$},
  Numerische Mathematik, 35 (1980), pp.~315--341,
  \url{https://doi.org/10.1007/BF01396415}.

\bibitem{olshanskii2014iterative}
{\sc M.~A. Olshanskii and E.~E. Tyrtyshnikov}, {\em Iterative methods for
  linear systems: {T}heory and applications}, SIAM, 2014.

\bibitem{Pavarino:1993}
{\sc L.~F. Pavarino}, {\em {Additive Schwarz methods for the $p$-version finite
  element method}}, Numerische Mathematik, 66 (1993), pp.~493--515,
  \url{https://doi.org/10.1007/BF01385709}.

\bibitem{masud2021new}
{\sc M.~M. Rana, V.~E. Howle, K.~Long, A.~Meek, and W.~Milestone}, {\em A new
  block preconditioner for implicit {Runge--Kutta} methods for parabolic {PDE}
  problems}, SIAM Journal on Scientific Computing, 43 (2021), pp.~S475--S495.

\bibitem{Rathgeber:2016}
{\sc F.~Rathgeber, D.~A. Ham, L.~Mitchell, M.~Lange, F.~Luporini, A.~T.~T.
  McRae, G.-T. Bercea, G.~R. Markall, and P.~H.~J. Kelly}, {\em {Firedrake:
  automating the finite element method by composing abstractions}}, ACM
  Transactions on Mathematical Software, 43 (2016), pp.~24:1--24:27,
  \url{https://doi.org/10.1145/2998441},
  \url{https://arxiv.org/abs/1501.01809}.

\bibitem{Schoeberl:2008}
{\sc J.~Sch\"oberl, J.~M. Melenk, C.~Pechstein, and S.~Zaglmayr}, {\em
  {Additive Schwarz preconditioning for $p$-version triangular and tetrahedral
  finite elements}}, IMA Journal of Numerical Analysis, 28 (2008), pp.~1--24,
  \url{https://doi.org/10.1093/imanum/drl046}.

\bibitem{southworth2022fast1}
{\sc B.~S. Southworth, O.~Krzysik, and W.~Pazner}, {\em Fast solution of fully
  implicit {R}unge-{K}utta and discontinuous {G}alerkin in time for numerical
  {PDE}s, {P}art {II}: {N}onlinearities and {DAE}s}, SIAM J. Sci. Comput., 44
  (2022), pp.~A636--A663, \url{https://doi.org/10.1137/21M1390438}.

\bibitem{southworth2022fast2}
{\sc B.~S. Southworth, O.~Krzysik, W.~Pazner, and H.~De~Sterck}, {\em Fast
  solution of fully implicit {R}unge-{K}utta and discontinuous {G}alerkin in
  time for numerical {PDE}s, {P}art {I}: {T}he linear setting}, SIAM J. Sci.
  Comput., 44 (2022), pp.~A416--A443, \url{https://doi.org/10.1137/21M1389742}.

\bibitem{staff2006preconditioning}
{\sc G.~A. Staff, K.-A. Mardal, and T.~K. Nilssen}, {\em Preconditioning of
  fully implicit {R}unge-{K}utta schemes for parabolic {PDE}s}, Modeling,
  Identification, and Control, 27 (2006), pp.~109--123.

\bibitem{vanlent2005}
{\sc J.~Van~Lent and S.~Vandewalle}, {\em Multigrid methods for implicit
  {Runge--Kutta} and boundary value method discretizations of parabolic
  {PDEs}}, SIAM Journal on Scientific Computing, 27 (2005), pp.~67--92,
  \url{https://doi.org/10.1137/030601144}.

\bibitem{vanka1986block}
{\sc S.~P. Vanka}, {\em Block-implicit multigrid solution of navier-stokes
  equations in primitive variables}, Journal of Computational Physics, 65
  (1986), pp.~138--158.

\bibitem{wanner1996solving}
{\sc G.~Wanner and E.~Hairer}, {\em Solving ordinary differential equations
  {II}}, Springer Berlin Heidelberg, 1996.

\end{thebibliography}
\end{document}